\documentclass[a4paper,12pt,reqno,oneside]{amsart}
\usepackage[utf8]{inputenc}
\usepackage{amsmath,amssymb}
\usepackage[numbers]{natbib}
\usepackage{bm}
\usepackage{eucal}
\usepackage{array}
\usepackage{booktabs}
\usepackage[font=small,labelfont=bf]{caption}
\usepackage{subfigure}
\usepackage{enumitem}
\usepackage{indentfirst}
\usepackage[a4paper,top=3cm,bottom=3cm,left=3cm,right=3cm,headsep=18pt]{geometry}
\usepackage{graphicx}
\usepackage{microtype}
\usepackage{randtext}
\usepackage{url}
\usepackage{hyperref}

\theoremstyle{plain}
\newtheorem{teo}{Theorem}[section]

\newtheorem{lem}{Lemma}[section]

\theoremstyle{definition}

\theoremstyle{remark}
\newtheorem*{obs}{Remark}

\numberwithin{equation}{section}
\newcolumntype{P}[2]{>{#1\arraybackslash}m{#2}}

\newlist{lteo}{enumerate}{1}
\setlist[lteo,1]{font=\upshape,label=(\roman*)}

\newlist{lprova}{enumerate}{1}
\setlist[lprova,1]{leftmargin=*,font=\upshape,label=(\roman*)}

\newcommand{\bbR}{\mathbb{R}}
\newcommand{\bbZ}{\mathbb{Z}}

\newcommand{\bbE}{\mathbb{E}}
\newcommand{\bbV}{\mathbb{V}}
\newcommand{\bbC}{\mathbb{C}}

\newcommand{\eps}{\varepsilon}
\newcommand{\V}{\text{Var}}
\newcommand{\Cov}{\text{Cov}}

\newcommand{\KK}{\mathcal{K}}
\newcommand{\vzero}{\mathbf{0}}
\newcommand{\cpr}{\stackrel{\scriptscriptstyle{\!P}}{\longrightarrow}}
\newcommand{\Ft}{\mathcal{F}_t}
\newcommand{\Gt}{\mathcal{F}_t}
\newcommand{\ocirc}{\mathbin{\rlap{$\mspace{2mu}\cdot$}\hbox{$\circ$}}}

\DeclareMathOperator{\Binomial}{Binomial}
\DeclareMathOperator{\EmpBox}{EmpBox}

\begin{document}

\title[Frog model on the complete graph]{Laws of large numbers for the frog model \texorpdfstring{\\}{} on the complete graph}

\author{Elcio Lebensztayn}
\address[E. Lebensztayn and M. A. Estrada]{Institute of Mathematics, Statistics and Scientific Computation\\
University of Campinas -- UNICAMP\\
Rua S\'{e}rgio Buarque de Holanda 651, 13083-859, Campinas, SP, Brazil.}
%\email[E. Lebensztayn]{\randomize{lebensztayn@ime.unicamp.br}}
\thanks{Elcio Lebensztayn is thankful to the National Council for Scientific and Technological Development -- CNPq, and Mario Andr\'{e}s Estrada is thankful to the Coordination for the Improvement of Higher Education Personnel -- CAPES, for financial support.
Both authors are also grateful to the S\~ao Paulo Research Foundation -- FAPESP (grant 2017/10555-0).}

\author{Mario Andr\'{e}s Estrada}
%\email[M. A. Estrada]{\randomize{marioestradalopez@gmail.com}}
%\thanks{}

\date{}



\begin{abstract}
The frog model is a stochastic model for the spreading of an epidemic on a graph, in which a dormant particle starts to perform a simple random walk on the graph and to awake other particles, once it becomes active.
We study two versions of the frog model on the complete graph with $N + 1$ vertices.
In the first version we consider, active particles have geometrically distributed lifetimes.
In the second version, the displacement of each awakened particle lasts until it hits a vertex already visited by the process.
For each model, we prove that as $N \to \infty$, the trajectory of the process is well approximated by a three-dimensional discrete-time dynamical system. 
We also study the long-term behavior of the corresponding deterministic systems.
\end{abstract}

\maketitle

\baselineskip=22pt

\section{Introduction}

We study stochastic systems whose agents are particles with random lifetimes, that move along the vertices of a finite graph.
These systems, referred to as \textit{frog models}, are commonly idealized with the motivation of modeling the spreading of a rumor or epidemic through a population.
Our main purpose is to prove limit theorems that allow us to identify the deterministic system (in discrete time) that approximates the trajectory of the process, when the graph size becomes large.

To define the models under study, for $N \geq 3$, let $\KK_{N+1}$ denote the complete graph with $N+1$ vertices.
At time zero, there is one particle at each vertex of $\KK_{N+1}$, all of them are sleeping, except for that placed at a fixed vertex of the graph.
Once wakened, particles perform independent simple random walks on $\KK_{N+1}$ in discrete time.
When a vertex with a sleeping particle is visited for the first time, that particle is activated and starts its own random walk.
However, each active particle has a random lifetime.
In the first model we consider, before each jump, active particles choose either to survive with probability $p$ or to disappear with probability $(1-p)$, independently of each other.
That is, upon being activated, a particle dies after a random number of jumps, having geometric distribution with parameter $(1-p)$.
We call this process the \textit{geometric model}.
In the second model we study, which we refer to as the \textit{nongeometric model}, each active particle survives up to the time it hits a vertex which has been visited before by the process.
In epidemic terms, the active particles are pathogenic agents, which move along the vertices of the graph (individuals).
Whenever a virus jumps onto a susceptible individual (an unvisited vertex), this individual becomes infected, and the virus duplicates.
In the geometric model, viruses have independent trajectories and lifetimes.
In the nongeometric model, once visited, a vertex activates an antivirus which kills every virus that tries to infect it in the future.
We study the fundamental question concerning how the proportion of visited vertices (infected individuals) evolves throughout the process.

Both systems of random walks on the complete graph are analyzed in \citet{RWSCG}.
For the geometric model, the authors prove that a critical parameter related to the final epidemic size equals $1/2$.
For the nongeometric model, the results are derived from computational analysis, simulations and mean field approximation.
In our work, motivated by the results and open questions presented there, we establish a law of large numbers which guarantees that the sample paths of the stochastic process stay close to the solution of a discrete dynamical system.
For the geometric model, the corresponding deterministic system is a discrete-time version of the classical epidemic model proposed by \citet{KKI,KKII}, whereas for the nongeometric model it is the dynamical system found by \citet{RWSCG} through a mean field approximation approach.
In addition, for both models, we study the long-term behavior of the associated deterministic systems.
The results are stated in Section~\ref{S: Main results}.

We point out that limit theorems of the type we derive are well-known for continuous-time Markovian processes; see \citet{Barbour}, \citet[Chapter~11]{EK} and \citet{DN}.
Regarding the frog model, \citet{EMCG} and \citet{EMIM} obtain limit theorems for the final outcome of different types of the continuous-time system on the complete graph.
\citet{MMM} deal with a discrete-time model also on the complete graph, in which a unique active particle is allowed to jump at each instant of time.
Differently from these three papers, in our models, several particles can jump simultaneously, so that the transitions may be abrupt.
With respect to other processes in discrete time, \citet{BP} prove limit theorems for a class of one-dimensional chain binomial models, and other models with similar characteristics.
As we will see in the sequel, the systems that we study are described by Markov chains taking values in $\bbZ^3$, whose general structures do not fulfill the conditions of the theorems stated in that paper
(since the functions that one uses to describe our density processes depend upon \(t\) and \(N\); see, in particular, Lemmas~\ref{L: G1}, \ref{L: G2}, \ref{L: NG1} and \ref{L: NG2}).
The techniques we develop here are inspired by the methods presented in \citet{BP}.

\subsection*{A brief overview on the frog model}

Besides the epidemic perspective, modeling combustion chemical reactions provides another interpretation for the frog model.
In this case, the systems are often formulated in continuous time; see, e.g., \citet{CQR}, \citet{RS}, and \citet{RS-IM}.

The model has been primarily studied on infinite graphs, such as the integer lattices~$\bbZ^d$, homogeneous trees and $d$-ary trees.
In the case when the frogs have geometrically distributed lifetimes, there exists a critical value of the parameter~$p$ below which the system dies out almost surely.
This critical phenomenon is an important topic of research; see \citet{PT}, \citet{FMS}, \citet{PTBT,NUB}, and references therein.
The issue of survival is also addressed for similar models on $\bbZ$ in \citet{BMZ} and \citet{RWSZ-JSP}.

When active frogs live forever, other fundamental questions are considered, such as:
(i) whether the root vertex of the graph is visited by frogs infinitely often; and
(ii) how the set of visited vertices grow and the cloud of particles move.
With respect to (i), \citet{TW} establish that the frog model starting with one particle per vertex is recurrent on $\bbZ^{d}$, for every $d \geq 1$.
A threshold result on $\bbZ^{d}$, with Bernoulli initial configuration, is derived by \citet{Serguei-FRE}.
Recurrence and transience on infinite trees was a prominent open question until recently.
\citet{HJJ-RT} establish that the frog model is recurrent on the binary tree, but it is transient on $d$-ary trees, for $d\geq 5$.
We also refer to \citet{HJJ-P}, \citet{JJ-CD} and \citet{Rosenberg-AT}.
The subjects in (ii) are investigated mainly for processes on $\bbZ^d$ (see, for instance, \citet{Shape,RS,HW}), and lately analyzed on $d$-ary trees (\citet{HJJ-IS}).

Having in mind the epidemic interpretation, it is natural to consider the frog model on finite graphs.
As we have already mentioned, \citet{RWSCG}, \citet{EMCG}, \citet{MMM} and \citet{EMIM} consider systems on the complete graph where active particles have a random lifetime, and study the fraction of visited vertices (infected individuals) by the process, as the size of the graph goes to $\infty$.
Also regarding the literature on finite graphs, one may assume that active frogs die after taking a fixed number~$\tau$ of jumps, and then obtain estimates for the smallest value $\tau$ could be, in order to guarantee that, with probability at least $1/2$, every vertex of the graph will be visited.
In the nice survey on the frog model written by \citet{Serguei-FIRW}, some results about this quantity are stated, for different types of graphs and initial configuration of one particle per vertex.
A related problem concerns the susceptibility of the graph, which is the minimal lifetime of the particles required for the process to visit all vertices, before dying out.
This variable is analyzed on finite $d$-ary trees by \citet{Hermon}, and on regular expanders and finite-dimensional tori by \citet{BFHM}.
Other statistic of interest is the cover time (time until all vertices are visited at least once, when active frogs have infinite lifetimes), recently studied on the complete graph (\citet{CDLLSY}) and on finite trees (\citet{Hermon}, \citet{HJJ-CT}).

\section{Main results}
\label{S: Main results}

\subsection{Geometric model}

For a realization of the frog model with geometric lifetimes on $\KK_{N+1}$, we define
\begin{equation*}
{\allowdisplaybreaks
\begin{aligned}
I_t &= \text{Number of unvisited vertices at time $t$}, \\
A_t &= \text{Number of actives particles at time $t$}, \\
D_t &= \text{Number of particles that have already died up to time $t$}.
\end{aligned}}%
\end{equation*}
Notice that the number of visited vertices at time $t$ is $V_t=N+1-I_t$, and the number of dead particles at time $t$ satisfies $D_t=V_t-A_t$.
Consequently, $I_t+A_t+D_t = N+1$ for every $t\geq0$.
For the simplicity of notation, we omit the dependence on $N$ of these random variables.

Of course, $\{(I_t, A_t, D_t)\}_{t \geq 0}$ is a discrete-time Markov chain (in fact, $\{(I_t, A_t)\}_{t \geq 0}$ is), such that $A_1=2$, $D_1=0$, and $I_1=N-1$ with probability $1$.
The following stochastic representation will be useful in the analysis of the process.
We think the actions of removing some particles and generating new ones as operations that occur at an intermediate instant between $t$ and $t+1$, and are registered only at time $t+1$.
Once the state of the process $(I_t,A_t,D_t)$ at time $t$ is known, we consider two auxiliary random variables $X_{t+1}$ and $Z_{t+1}$, distributed as
\begin{equation}
\label{F: Aux-G}
X_{t+1} \sim \Binomial(A_t, p) \quad \text{and} \quad
Z_{t+1} \sim \Binomial\left(X_{t+1}, \dfrac{I_t}{N}\right).
\end{equation}
The random variables $X_{t+1}$ and $Z_{t+1}$ stand respectively for the number of particles that survive, and the number of the survivor particles that decide to choose new vertices. 
Then, the state of the process at time $t+1$ is given by
\begin{equation}
\label{F: Trans-G}
{\allowdisplaybreaks
\begin{aligned}
I_{t+1} &\sim \EmpBox(Z_{t+1}, I_t), \\[0.1cm]
A_{t+1} &= X_{t+1}+I_t-I_{t+1}, \\[0.1cm]
D_{t+1} &= N+1-I_{t+1}-A_{t+1}.
\end{aligned}}%
\end{equation}
Here, $\EmpBox(Z_{t+1}, I_t)$ is the distribution of the number of empty urns when $Z_{t+1}$ balls are randomly distributed in $I_t$ urns, so that the difference $I_t-I_{t+1}$ represents the number of newly activated particles.
Further details on the $\EmpBox$ distribution are presented in Section~\ref{SS: Auxiliary results}.

Our first result establishes that, as $N \to \infty$, the stochastic process scaled by $N + 1$ behaves approximately as a discrete-time deterministic system.
For $t \geq 0$, we define
\begin{equation}
\label{F: VA}
\eta_t = (i_t,a_t,d_t) = 
\left(\frac{I_t}{N + 1}, \frac{A_t}{N + 1}, \frac{D_t}{N + 1}\right).
\end{equation}
Notice that, from the epidemic perspective, the fraction \(i_t\) corresponds to the proportion of individuals of the population that have not been infected up to time~$t$.
Now let $\xi_t = (\iota_t,\alpha_t,\delta_t)$ be three-dimensional dynamical system described by the equations
\begin{equation}
\label{F: SDG}
{\allowdisplaybreaks
\begin{aligned}
&\iota_{t+1} = \iota_t \, e^{-p \alpha_t},\\[0.2cm]
&\alpha_{t+1} = p \, \alpha_t+\iota_t \, (1-e^{-p\alpha_t}),\\[0.2cm]
&\delta_{t+1} = \delta_t+(1-p) \, \alpha_t,\\[0.2cm]
&\iota_0 = \dfrac{N}{N+1}, \; \alpha_0 = \dfrac{1}{N+1}, \; \delta_0 = 0.
\end{aligned}}%
\end{equation}
Again the dependence on $N$ is omitted from the notation, except in cases where it is necessary.
The system~\eqref{F: SDG} corresponds to a discrete-time Kermack--McKendrick model of epidemics, with removal and infectious rates both equal to $p$. 
We refer to \citet{KKI,KKII} and \citet[Chapter~9]{HP} for more details.

\begin{teo}
\label{T: LLN-G}
Consider the frog model on $\KK_{N+1}$ with geometric lifetimes.
For every $t\geq 0$, we have that
\begin{equation}
\label{F: LLN-G}
\eta_t - \xi_t \cpr \vzero = (0, 0, 0) \quad \text{as } N \to \infty,
\end{equation}
where $\cpr$ denotes convergence in probability.
\end{teo}

Theorem~\ref{T: LLN-G} is proved by mathematical induction on $t$.
Roughly speaking, we have that the conditional expected value of the next state of the scaled stochastic process given the current state behaves approximately as the corresponding state of the deterministic model.
Consequently, if the stochastic and the deterministic systems are close to each other at a given instant of time $t$, then the same occurs at time $t+1$.
The proof is inspired by the techniques of \citet{BP}, whose results can not be directly applied to our models.

Now we examine the long-term behavior of the sequence $\{\xi_t\} = \{\xi_t^{(N)}\}$ given by the deterministic system~\eqref{F: SDG}, first as $t$ increases and then as $N$ increases.
Let $\phi: (0, 1) \rightarrow (0, \infty)$ be the function given by
\[ \phi(p) = \dfrac{p}{1-p}, \, p \in (0, 1). \]
Let $W_0$ denote the principal branch of the so-called Lambert $W$ function (which is the multivalued inverse of the function $ x \mapsto x \, e^x $).
More details about this function can be found in \citet{CGHJK}.

\begin{teo}
\label{T: ABG}
For every $p \in (0, 1)$ and $N \geq 3$, the following limits exist:
\[ \iota_{\infty}^{(N)} = \displaystyle\lim_{t\to \infty} \iota_t^{(N)}, \,
\delta_{\infty}^{(N)} = \displaystyle\lim_{t\to \infty} \delta_t^{(N)}, \text{ and } \,
\displaystyle\lim_{t\to \infty} \alpha_t^{(N)} = 0, \]
with $\iota_{\infty}^{(N)} + \delta_{\infty}^{(N)} = 1$.
Furthermore, there exists $\iota_{\infty}=\displaystyle\lim_{N\to \infty}\iota_{\infty}^{(N)}$, which is given by
\begin{equation*}
\iota_{\infty}=
\left\{
\begin{array}{cl}
1 &\text{if } p\leq 1/2,\\[0.2cm]
-\dfrac{1}{\phi(p)} \, W_0(-\phi(p) \, e^{-\phi(p)}) &\text{if } p> 1/2.
\end{array}	\right.
\end{equation*}
\end{teo}

The graph of $\iota_{\infty}$ as a function of $p$ is presented in Figure~\ref{Fig: IFinal}.

\vspace{0.2cm}

\begin{figure}[!htb]
\centering
\includegraphics[scale=0.6]{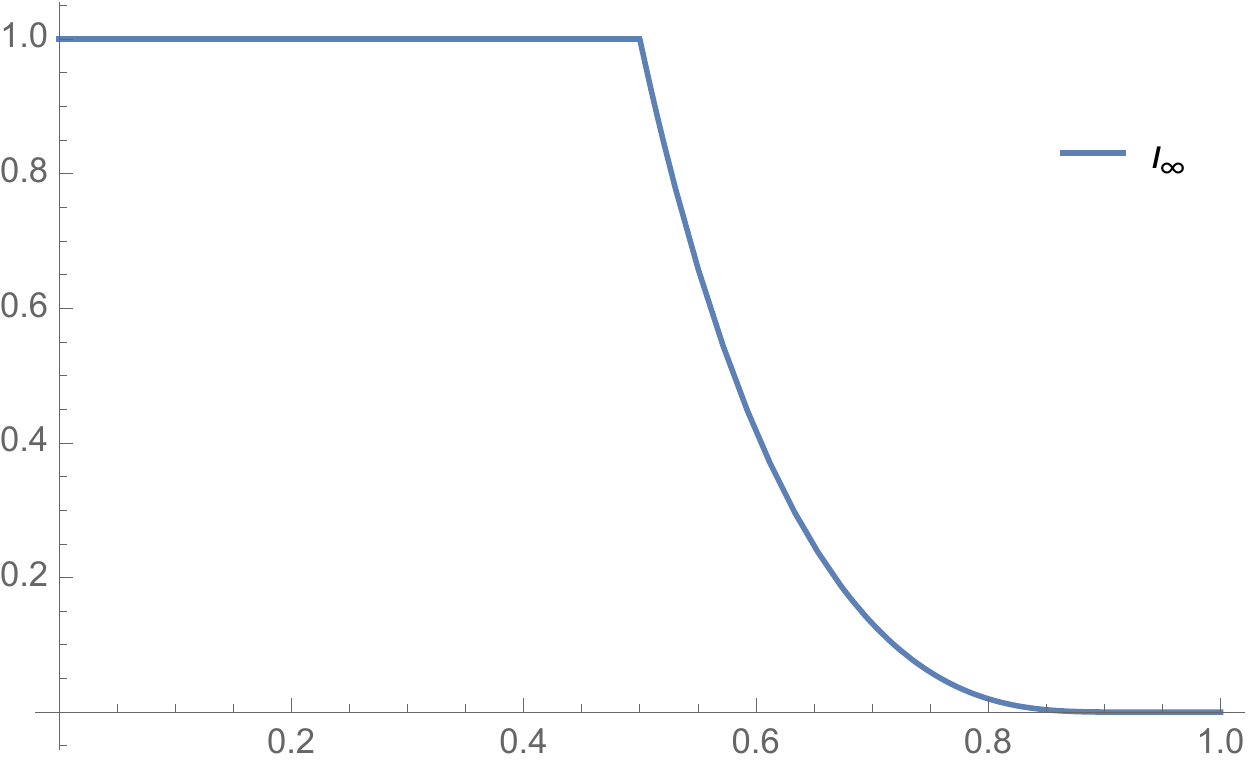}
\caption{Function $\iota_{\infty} = \iota_{\infty}(p)$.}
\label{Fig: IFinal}
\end{figure}

\begin{obs}
For a fixed $p\leq 1/2$, $\iota_{\infty}=1$ is the unique fixed point of the function
\begin{equation}
\label{F: Func}
\tau(x) = \exp\{-\phi(p)(1-x)\}
\end{equation}
in the closed interval $[0, 1]$.
On the other hand, if $p>1/2$, then $\tau$ has two fixed points in $[0, 1]$, and $\iota_{\infty}$ is the fixed point strictly less than~$1$.
The cases $p=0.4$ and $p= 0.6$ are illustrated in Figure~\ref{Fig: Func}.
\end{obs}

\begin{figure}[!htb]
\centering
\subfigure[$p=0.4$]{\includegraphics[width=0.4\textwidth]{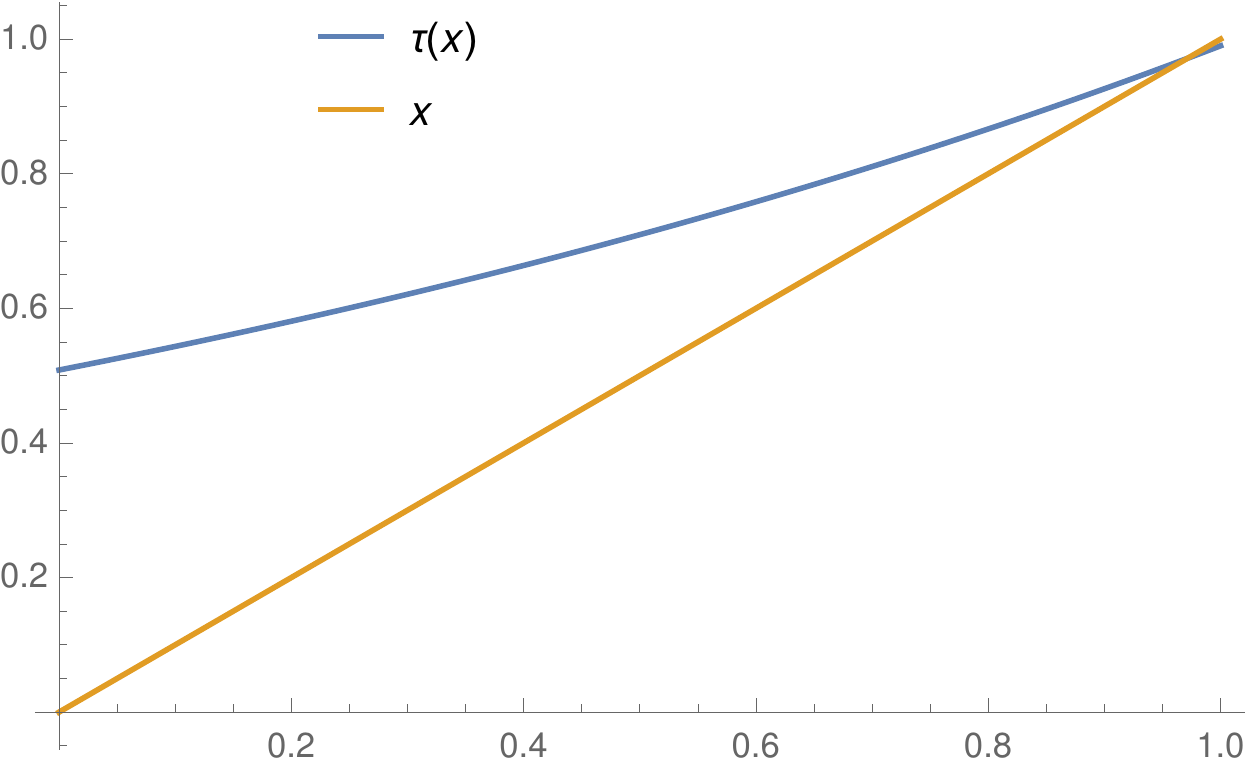}\label{Fig: Func-a}}
\qquad
\subfigure[$p=0.6$]{\includegraphics[width=0.4\textwidth]{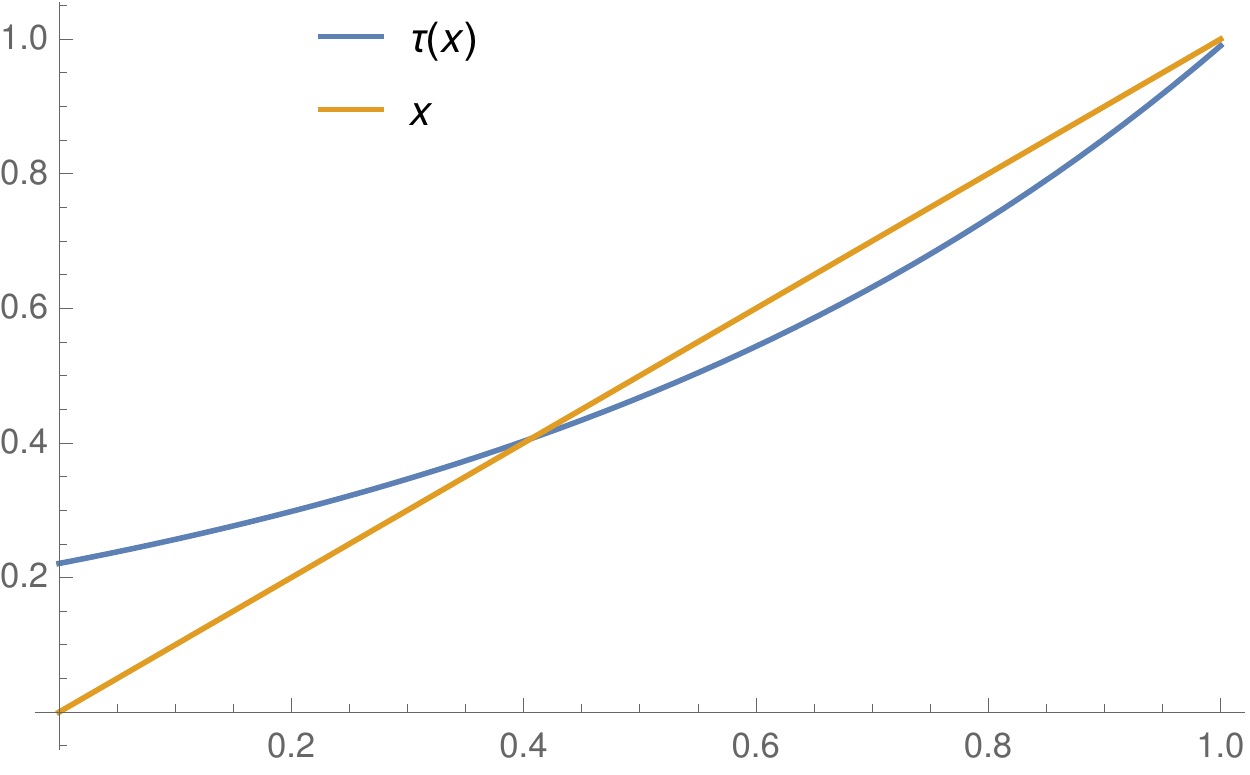}\label{Fig: Func-b}}
\caption{Behavior of the function $\tau$.}
\label{Fig: Func}
\end{figure}

The limiting behavior of the deterministic system \eqref{F: SDG} exhibited in Theorem~\ref{T: ABG} is in agreement with the result proved by \citet{RWSCG} for the model with geometric lifetimes.
Indeed, this result states that the critical value of $p$ below which the final proportion of visited vertices converges in distribution to zero equals $1/2$.

The proofs of Theorems~\ref{T: LLN-G} and \ref{T: ABG} are presented respectively in Sections~\ref{SS: Proof LLN-G} and \ref{SS: Proof ABG}.

\subsection{Nongeometric model}
\label{SS: NGM}

For the model with nongeometric lifetimes, we adopt the same notation as before:
$I_t$ is the number of unvisited vertices, $A_t$ is the number of actives particles, and $D_t$ is the number of dead particles at time $t$.
The state of the process (scaled by $N+1$) at time $t$ is summarized by the random vector $\eta_t$, as defined in~\eqref{F: VA}.

We also use a stochastic representation to describe how the process evolves.
Here every active particle always jumps, surviving only if it reaches an unvisited vertex.
Hence, given the state of the process $(I_t,A_t,D_t)$ at time $t$, we consider a single auxiliary random variable
\begin{equation}
\label{F: Aux-NG}
Z_{t+1}\sim \Binomial\left(A_t, \dfrac{I_t}{N}\right),
\end{equation}
which represents the number of active particles that choose new vertices and survive.
Consequently, the state of the process at time $t+1$ is given by
\begin{equation}
\label{F: Trans-NG}
{\allowdisplaybreaks
\begin{aligned}
I_{t+1} &\sim \EmpBox(Z_{t+1}, I_t), \\[0.1cm]
A_{t+1} &= Z_{t+1}+I_t-I_{t+1}, \\[0.1cm]
D_{t+1} &= N+1-I_{t+1}-A_{t+1}.
\end{aligned}}%
\end{equation}
The essence of the two stochastic systems is the set of random variables that help their mathematical modeling. 
Notice, however, that there are fundamental differences between them.
In the geometric model, each active particle determines its survival through the toss of a coin, which is described by the random variable $X_{t+1}$.
In both models, the random variable $Z_{t+1}$ represents the number of particles that choose jumping to unvisited vertices, but for the geometric model the underlying binomial distribution applies to a thinner group of particles (modeled by $X_{t+1}$).
Finally, the most important difference lies in the formula for $A_{t+1}$: for the geometric model, it has as addends $X_{t+1}$ (number of survivor particles) and $I_t-I_{t+1}$ (number of new particles introduced in the system), whereas for the nongeometric model the first addend is $Z_{t+1}$.

For the nongeometric model, we denote the deterministic counterpart of $\eta_t$ by $\tilde{\xi}_t = (\tilde{\iota}_t,\tilde{\alpha}_t,\tilde{\delta}_t)$, which is given by the following discrete-time dynamical system:
\begin{equation}
\label{F: SDNG}
{\allowdisplaybreaks
\begin{aligned}
&\tilde{\iota}_{t+1}=\tilde{\iota}_t \, e^{-\tilde{\alpha}_t},\\[0.2cm]
&\tilde{\alpha}_{t+1}=\tilde{\iota}_t \, (\tilde{\alpha}_t+1-e^{-\tilde{\alpha}_t}),\\[0.2cm]
&\tilde{\delta}_{t+1}=\tilde{\delta}_t+\tilde{\alpha}_t\left(1-\tilde{\iota}_t\right),\\[0.2cm]
&\tilde{\iota}_0=\dfrac{N}{N+1}, \; \tilde{\alpha}_0=\dfrac{1}{N+1}, \; \tilde{\delta}_0=0.
\end{aligned}}%
\end{equation}
This dynamical system is derived in \citet[Section~3.3]{RWSCG} for the nongeometric model, through a mean field approximation approach. 
In this paper, the authors underline the remarkable resemblance in the evolution of the stochastic and the deterministic systems, for various values of the degree of the graph (Figures~$1$ and $2$ in the paper).
This is fully justified by the law of large numbers for the trajectory of the process, which we now establish.

\begin{teo}
\label{T: LLN-NG}
Consider the frog model on $\KK_{N+1}$ with nongeometric lifetimes.
For every $t\geq 0$, we have that
\[ \eta_t - \tilde{\xi}_t \cpr \vzero \quad \text{as } N \to \infty. \]
\end{teo}

The analysis of the deterministic system~\eqref{F: SDNG} is more difficult than that of~\eqref{F: SDG}, mainly due to the presence of the variable $\tilde{\iota}_t$ in place of $p$, in the third equation.
For the sake of completeness, we summarize the main results for this system derived by~\citet[Theorem~3.2]{RWSCG}, adapting to our notation.
The authors prove that, for every fixed $N$, the sequence $\{ \tilde{\alpha}_t \}$ follows the
pattern
\[ \tilde{\alpha}_0 < \tilde{\alpha}_1 < \cdots < \tilde{\alpha}_{M-1} \leq \tilde{\alpha}_M
> \tilde{\alpha}_{M+1} > \tilde{\alpha}_{M+2} > \cdots, \]
for some $M = M(N)$, with $ \displaystyle\lim_{t \to \infty} \tilde{\alpha}_t = 0 $.
In addition, they establish the existence of the following limits:
\[ \tilde{\iota}_{\infty}^{(N)} = \displaystyle\lim_{t \to \infty} \tilde{\iota}_t \in (0.17, 0.18) 
\quad \text{and} \quad
\tilde{\delta}_{\infty}^{(N)} = \displaystyle\lim_{t \to \infty} \tilde{\delta}_t \in (0.82, 0.83). \]
To study the sequence $\{\tilde{\iota}_{\infty}^{(N)}\}$, we use a mathematical software to run the deterministic system for various values of $N$, until it stabilizes.
Then, we construct a plot of the limiting value of the sequence $\{\tilde{\iota}_t^{(N)}\}$ for large $t$, against $N$.
The plot is shown in Figure~\ref{Fig: ANG}.
This suggests that the sequence $\{\tilde{\iota}_{\infty}^{(N)}\}$ is increasing in $N$, and tends to the value $0.174545$ as $N \to \infty$.
Thus, it is natural to conjecture that, for the nongeometric model, the final proportion of visited vertices converges in probability to $0.825455$ as $N \to \infty$.

\begin{figure}[!htb]
\begin{center}
\includegraphics[width=0.75\textwidth]{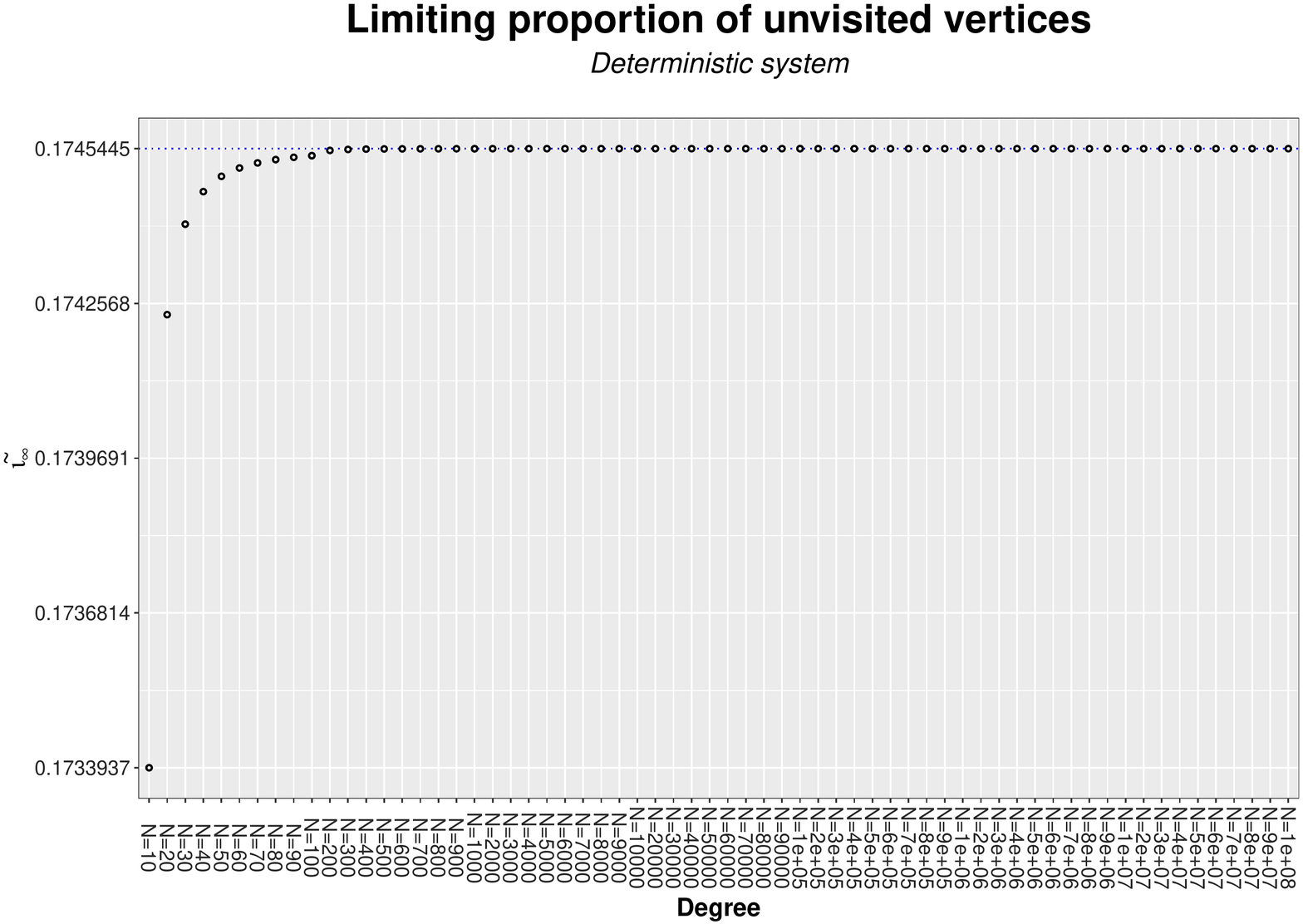}
\caption{Limiting values of the deterministic trajectory $\tilde{\iota}_t^{(N)}$ 
for different values of $N$.}
\label{Fig: ANG}
\end{center}
\end{figure}

\begin{obs}
There is a close relationship between the frog model with nongeometric lifetimes and stochastic models for the transmission of a rumor within a closed finite population.
Roughly speaking, we have the following correspondence:
\begin{lprova}
\item an inactive particle being awakened $=$ an ignorant individual becoming a spreader after hearing the rumor from a spreader;
\item an active particle jumping to a visited vertex $=$ a spreader contacting somebody who has already heard the rumor (another spreader or a stifler individual).
\end{lprova}
\citet{RWRT} present results on the asymptotic behavior and a joint construction of a continuous-time frog model on the complete graph and a general version of the stochastic rumor model introduced by~\citet{MT}.
Using similar ideas, we conclude that, except by a minor modification, our frog model with nongeometric lifetimes on $\KK_{N+1}$ corresponds to a discrete-time stochastic Maki--Thompson model, in which at each instant of time all spreaders simultaneously make contact with someone else in the population.
In this analogy, $I_t$, $A_t$ and $D_t$ represent respectively the number of ignorant, spreader and stifler individuals at time~$t$.
We point out that, in the original Maki--Thompson model, just one spreader makes a contact at each step.
For the classical model, \citet{Sudbury} proved that, as the population size tends to $\infty$, the ultimate proportion of ignorants converges in probability to the number $-W_0(- 2 \, e^{-2})/2 \approx 0.203188$.
The previous numerical analysis indicates that, for the Maki--Thompson model in which all spreaders act simultaneously, the limiting fraction of ignorants in the population would be equal to $0.174545$.
For more on the definitions and limit theorems for stochastic rumour models, we refer to \citet[Chapter~5]{DG}, \citet{LDPMT}, and \citet{LTRM}.
\end{obs}

\section{Proofs}

\subsection{Auxiliary results}
\label{SS: Auxiliary results}

We present here some definitions and results that are useful in the proofs of Theorems~\ref{T: LLN-G} and \ref{T: LLN-NG}.

We start off with the classical occupancy problem, referring the interested reader to \citet[Section~4 of Chapter~10]{JKK} for more details.
Consider the random placement of $b$ balls into $c$ boxes, in such a way that the balls are thrown independently and uniformly into the boxes.
Let $X$ denote the number of empty boxes after the balls have been distributed.
Then the probability mass function of $X$ is given by
\begin{equation*}
P(X = x) = \sum_{i=0}^{c-x}(-1)^{i}\binom{x+i}{i}\binom{c}{x+i}
\left(1 - \dfrac{x+i}{c}\right)^b,
\, x = 0, 1, \dots, c.
\end{equation*}
We write $X \sim \EmpBox(b, c)$.
In the sequel, we will use the formulas for the expectation and the variance of $X$, which are given by
{\allowdisplaybreaks
\begin{align}
E(X) &= c \left(\frac{c-1}{c}\right)^b, \, \text{ and}\label{F: EspEB}\\[0.1cm]
\V(X) &= c \, (c-1) \left(\frac{c-2}{c}\right)^b + 
c \left(\frac{c-1}{c}\right)^b - 
c^{2} \left(\frac{c-1}{c}\right)^{2 b}.\label{F: VarEB}
\end{align}}%

In the proofs, we will also use the following well-known result about the probability generating function of the binomial distribution.
Let $Y$ be a random variable with $\Binomial(n,p)$ distribution, $n \geq 1$, $p \in (0, 1)$.
Then, for every $s \in \bbR$,
{\allowdisplaybreaks
\begin{align}
E(s^Y) &= \left[1-p(1-s)\right]^n, \, \text{ and}\label{F: FGP-Bin1}\\[0.1cm]
E(Y s^Y) &= n p s \left[1-p(1-s)\right]^{n-1}.\label{F: FGP-Bin2}
\end{align}}%

%\begin{proof}
%Note that, by the Newton's Binomial Formula,
%\begin{align*}
%E(s^Y)&=\sum_{x=0}^{n}s^x\binom{n}{x}p^x(1-p)^{n-x}=\sum_{x=0}^{n}\binom{n}{x}(ps)^x(1-p)^{n-x}\\[0.2cm]
%&=\left[ps+1-p\right]^n=\left[1-p(1-s)\right]^n.
%\end{align*}
%Furthermore,
%\begin{align*}
%E(Ys^Y)&=sE(Ys^{Y-1})=s\dfrac{\partial}{\partial s}E(s^Y)\\[0.2cm]
%&=nps\left[1-p(1-s)\right]^{n-1}.
%\end{align*}	
%\end{proof}

\subsection{Proof of Theorem~\ref{T: LLN-G}}
\label{SS: Proof LLN-G}

The proof of \eqref{F: LLN-G} proceeds by mathematical induction on $t$.
Let
\begin{equation*}
\Ft = \sigma(\{\left(I_s, A_s, D_s\right): 0 \leq s \leq t \}), \, t \geq 0,
\end{equation*}
be the natural filtration generated by the process $\{(I_t, A_t, D_t)\}_{t \geq 0}$.
In words, $\Ft$ represents all the information available from the process at time $t$.
For $t \geq 0$, let us define
$\bbE(\cdot)=E(\,\cdot\,|\Ft)$, 
$\bbV(\cdot)=\V(\,\cdot\,|\Ft)$, and
$\bbC(\cdot,\ocirc)=\Cov(\,\cdot\,,\,\ocirc\,|\Ft)$.

In the first step of the proof of Theorem~\ref{T: LLN-G}, we write down closed formulas for the conditional expectation and the conditional variance of the next state of the Markov chain $\{(I_t, A_t, D_t)\}$, given the current state:
{\allowdisplaybreaks
\begin{align*}
\bbE\left((I_{t+1}, A_{t+1}, D_{t+1})\right) &:= \left(\bbE(I_{t+1}), \bbE(A_{t+1}), \bbE(D_{t+1})\right), 
\, \text{ and}\\[0.2cm]
\bbV\left((I_{t+1}, A_{t+1}, D_{t+1})\right) &:= \left(\bbV(I_{t+1}), \bbV(A_{t+1}), \bbV(D_{t+1})\right).
\end{align*}}%
For these computations, we use the expressions for the expected value and variance of the $\EmpBox$ distribution, and the results for the binomial distribution, which are presented in Section~\ref{SS: Auxiliary results}.

In the second step of the proof, we assume that \eqref{F: LLN-G} holds true for some $t \geq 0$, and prove that
\begin{equation}
\label{F: BHI}
\bbE(\eta_{t+1})-\xi_{t+1} \cpr \vzero \quad \text{and} \quad
\bbV(\eta_{t+1}) \cpr \vzero \quad \text{as } N \to \infty.
\end{equation}
That is, under the induction hypothesis, we have that the conditional expected value of the next state of the scaled process given the current state is well approximated by the corresponding state of the deterministic model, and the vector of conditional variances is close to $\vzero$.
Once \eqref{F: BHI} is established, we show the inductive step by using fundamental results of probability theory.
The two steps just described are formalized in Lemmas~\ref{L: G1} to~\ref{L: G4}, which are stated in the sequel.

\begin{lem}
\label{L: G1}
For every $t \geq 0$, the conditional expectations of $I_{t+1}$, $A_{t+1}$ and $D_{t+1}$ given $\Ft$ are, respectively,
{\allowdisplaybreaks
\begin{align*}
\bbE(I_{t+1})&= I_t\left(1-\frac{p}{N}\right)^{A_t},\\[0.2cm]
\bbE(A_{t+1})&= p A_t+I_t\left(1-\left(1-\frac{p}{N}\right)^{A_t}\right),\\[0.2cm]
\bbE(D_{t+1})&= D_t+(1-p)A_t.
\end{align*}}%
\end{lem}

\begin{proof}
First, it follows from~\eqref{F: Aux-G} and \eqref{F: FGP-Bin1} that, for every $s \in \bbR$,
\begin{equation*}
\bbE(s^{Z_{t+1}}|X_{t+1}) = {E}(s^{Z_{t+1}}|X_{t+1},\Ft)
= \left(1-\dfrac{I_t}{N}(1-s)\right)^{X_{t+1}},
\end{equation*}
whence, taking again the expected value,
\begin{equation}
\label{F: FGM-Z-G}
\bbE(s^{Z_{t+1}}) = \left(1-\dfrac{p \, I_t}{N}(1-s)\right)^{A_t}.
\end{equation}
Now recalling~\eqref{F: Trans-G} and applying~\eqref{F: EspEB}, we obtain that
\begin{equation*}
\bbE(I_{t+1}|Z_{t+1}) = I_t\left(\dfrac{I_t-1}{I_t}\right)^{Z_{t+1}}.
\end{equation*}
Therefore,
\begin{equation*}
\bbE(I_{t+1}) = I_t\left(1-\dfrac{p \, I_t}{N}\cdot\dfrac{1}{I_t}\right)^{A_t}
= I_t\left(1-\dfrac{p}{N}\right)^{A_t}.
\end{equation*}
Since $A_{t+1} = X_{t+1}+I_t-I_{t+1}$ and $D_{t+1} = N+1-A_{t+1}-I_{t+1}$, we have that
{\allowdisplaybreaks
\begin{align*}
\bbE(A_{t+1}) &= pA_t+I_t\left(1-\left(1-\dfrac{p}{N}\right)^{A_t}\right), \, \text{ and}\\[0.2cm]
\bbE(D_{t+1}) &= N+1-pA_t-I_t=D_t+(1-p)A_t,
\end{align*}}%
as desired.
\end{proof}

\begin{lem}
\label{L: G2}
For every $t \geq 0$, the conditional variances of $I_{t+1}$, $D_{t+1}$ and $A_{t+1}$ given~$\Ft$ are, respectively,
{\allowdisplaybreaks
\begin{align*}
\bbV(I_{t+1})&= I_t \, \biggl(\left(I_t-1\right)\left(1-\frac{2p}{N}\right)^{A_t}-I_t\left(1-\frac{p}{N}\right)^{2A_t}+\left(1-\frac{p}{N}\right)^{A_t}\biggr),\\[0.2cm]
\bbV(D_{t+1})&= A_t \, p(1-p),\\[0.1cm]
\bbV(A_{t+1})&= \bbV(I_{t+1})+\bbV(D_{t+1}) + 2 \, p \, A_t \, I_t \left(1-\frac{p}{N}\right)^{A_t}\left(\frac{1-p}{N-p}\right).
\end{align*}}%
\end{lem}

\begin{proof}
From~\eqref{F: EspEB} and \eqref{F: VarEB}, we have that
\begin{gather*}
\bbE\left(I_{t+1}|Z_{t+1}\right) = I_t\left(\dfrac{I_t-1}{I_t}\right)^{Z_{t+1}} \, \text{ and}\\[0.2cm]
\V(I_{t+1}|Z_{t+1},\Ft) = I_t(I_t-1)\left(\dfrac{I_t-2}{I_t}\right)^{Z_{t+1}}+I_t\left(\dfrac{I_t-1}{I_t}\right)^{Z_{t+1}}-I_t^2\left(\dfrac{I_t-1}{I_t}\right)^{2Z_{t+1}}.
\end{gather*}
Hence, using the law of total variance and \eqref{F: FGM-Z-G}, we conclude that
{\allowdisplaybreaks
\begin{align*}
\bbV(I_{t+1})&= I_t(I_t-1)\left(1-\dfrac{2p}{N}\right)^{A_t}+I_t\left(1-\dfrac{p}{N}\right)^{A_t}-I_t^2 \, E\biggl(\left(\dfrac{I_t-1}{I_t}\right)^{2Z_{t+1}}\biggm|\Ft\biggr)+{}\\[0.2cm]
&\phantom{=}{}+I_t^2 \, E\biggl(\left(\dfrac{I_t-1}{I_t}\right)^{2Z_{t+1}}\biggm|\Ft\biggr)-I_t^2\left(1-\dfrac{p}{N}\right)^{2A_t}\\[0.2cm]
&= I_t(I_t-1)\left(1-\dfrac{2p}{N}\right)^{A_t}+I_t\left(1-\dfrac{p}{N}\right)^{A_t}-I_t^2\left(1-\dfrac{p}{N}\right)^{2A_t}.
\end{align*}}%

Moreover, as $D_{t+1} = N+1-X_{t+1}-I_t$, we get
\[ \bbV(D_{t+1}) = \bbV(X_{t+1}) = A_t \, p(1-p). \]

In order to compute the conditional variance of $A_{t+1}$ given $\Ft$, recall that $A_{t+1} = X_{t+1}+I_t-I_{t+1}$.
Thus,
\[ \bbV(A_{t+1})=\bbV(I_{t+1})+\bbV(D_{t+1})- 2 \, \bbC(I_{t+1}, X_{t+1}). \]
To finish the proof, it suffices to show that
\begin{equation}
\label{F: CovIX}
\bbC(I_{t+1}, X_{t+1}) = -p \, A_t \, I_t \left(1-\frac{p}{N}\right)^{A_t}\left(\frac{1-p}{N-p}\right).
\end{equation}
To compute the conditional covariance in \eqref{F: CovIX}, first notice that
{\allowdisplaybreaks
\begin{align*}
\bbE(I_{t+1}|X_{t+1})&=\bbE(\bbE(I_{t+1}|X_{t+1},Z_{t+1})|X_{t+1})\\[0.2cm]
&=\bbE\biggl(I_t\left(\dfrac{I_t-1}{I_t}\right)^{Z_{t+1}}\biggm|X_{t+1}\biggr)%\\[0.2cm]
%&= I_t\bbE\biggl(\left(\dfrac{I_t-1}{I_t}\right)^{Z_{t+1}}\biggm|X_{t+1}\biggr)\\[0.2cm]
= I_t\left(1-\dfrac{1}{N}\right)^{X_{t+1}}.
\end{align*}}%
Hence, using formula \eqref{F: FGP-Bin2}, we have that
{\allowdisplaybreaks
\begin{align*}
\bbE(I_{t+1} X_{t+1})&= \bbE(X_{t+1} \, \bbE(I_{t+1}|X_{t+1}))\\[0.2cm]
&= \bbE\biggl(X_{t+1} \, I_t\left(1-\dfrac{1}{N}\right)^{X_{t+1}}\biggr)
= p \, A_t \, I_t \left(1-\dfrac{1}{N}\right)\left(1-\dfrac{p}{N}\right)^{A_t-1}.
\end{align*}}%
Consequently,
{\allowdisplaybreaks
\begin{align*}
\bbC(I_{t+1}, X_{t+1}) &= \bbE(I_{t+1} X_{t+1}) - \bbE(I_{t+1}) \, \bbE(X_{t+1})\\[0.2cm]
&= p \, A_t \, I_t \left(1-\dfrac{p}{N}\right)^{A_t}
\left[\left(1-\dfrac{p}{N}\right)^{-1}\left(1-\dfrac{1}{N}\right) - 1\right]\\[0.2cm]
&= -p \, A_t \, I_t \left(1-\dfrac{p}{N}\right)^{A_t}\left(\dfrac{1-p}{N-p}\right).
\end{align*}}%
Thus, \eqref{F: CovIX} holds and the proof is complete.
\end{proof}

With the previous two lemmas in hand, we may now establish two preparatory results for the proof of Theorem~\ref{T: LLN-G}.

\begin{lem}
\label{L: G3}
Suppose that $\eta_t - \xi_t \cpr \vzero$ as $N \to \infty$, for some $t \geq 0$.
Then,
\[ \bbE(\eta_{t+1}) - \xi_{t+1} \cpr \vzero \quad \text{as } N \to \infty. \]
\end{lem}

\begin{proof}
Assume that $\eta_t-\xi_t \cpr \vzero$ for some $t \geq 0$.
Given $N \geq 3$ and $p \in [0, 1]$, we define the function
\begin{equation*}
\psi(i,a)=\psi_{N,p}(i,a)=i\left(1-\dfrac{p}{N}\right)^{(N+1)a}, \, (i,a)\in[0, 1]^2.
\end{equation*}
By Lemma~\ref{L: G1}, we have that
\begin{equation*}
\bbE(i_{t+1}) = i_t\left(1-\dfrac{p}{N}\right)^{(N+1)a_t}
= \psi(i_t,a_t).
\end{equation*}
Then, using the definition of $\iota_{t+1}$ given in~\eqref{F: SDG} and triangle inequality,
\begin{equation}
\label{F: DI}
{\allowdisplaybreaks
\begin{aligned}
\left\lvert \bbE(i_{t+1})-\iota_{t+1} \right\rvert&=
\left\lvert \psi(i_t,a_t)-\iota_t \, e^{-p\alpha_t} \right\rvert\\[0.1cm]
&\leq
\left\lvert \psi(i_t,a_t)-\psi(\iota_t,\alpha_t) \right\rvert
+\left\lvert \iota_t \, e^{-p\alpha_t}-\psi(\iota_t,\alpha_t) \right\rvert.
\end{aligned}}%
\end{equation}
But for every $N\geq 3$, $p \in [0, 1]$ and $(i,a)\in[0, 1]^2$,
{\allowdisplaybreaks
\begin{align*}
\left\lvert \dfrac{\partial\psi(i,a)}{\partial i} \right\rvert&=
\left\lvert \left(1-\dfrac{p}{N}\right)^{(N+1)a} \right\rvert \leq 1, \quad \text{and}\\[0.2cm]
\left\lvert \dfrac{\partial\psi(i,a)}{\partial a} \right\rvert&=
\left\lvert i\left(1-\dfrac{p}{N}\right)^{(N+1)a}(N+1) \, \log\left(1-\dfrac{p}{N}\right) \right\rvert\\[0.2cm]
&\leq -(N+1) \, \log\left(1-\dfrac{1}{N}\right) \leq 2.
\end{align*}}%
This implies that $\psi$ is a Lipschitz continuous function in $[0, 1]^2$, with Lipschitz constant at most equal to~$\sqrt{5}$.
Consequently,
\begin{equation}
\label{F: DL}
\left\lvert \psi(i_t,a_t)-\psi(\iota_t,\alpha_t) \right\rvert
\leq \sqrt{5} \left\lVert (i_t,a_t)-(\iota_t,\alpha_t) \right\rVert \cpr 0,
\end{equation}
as $N\to \infty$. 
In addition, standard calculus shows that
\begin{equation}
\label{F: DF}
{\allowdisplaybreaks
\begin{aligned}
\left\lvert \iota_t \, e^{-p\alpha_t}-\psi(\iota_t,\alpha_t) \right\rvert
%&\leq \left\lvert e^{-p\alpha_t}-\left(1-\dfrac{p}{N}\right)^{(N+1)\alpha_t} \right\rvert\\[0.2cm]
&\leq e^{-p\alpha_t}-\left(1-\dfrac{p}{N}\right)^{(N+1)\alpha_t}\\[0.2cm]
&\leq e^{-p}-\left(1-\dfrac{p}{N}\right)^{(N+1)} \rightarrow 0,
\end{aligned}}%
\end{equation}
as $N\to \infty$.
From \eqref{F: DI}, \eqref{F: DL} and \eqref{F: DF}, it follows that
\begin{equation}
\label{F: ConvI-G}
\bbE(i_{t+1})-\iota_{t+1} \cpr 0 \quad \text{as } N \to \infty.
\end{equation}

Now, using Lemma~\ref{L: G1} and formula \eqref{F: SDG}, we have
{\allowdisplaybreaks
\begin{align*}
\bbE(d_{t+1})-\delta_{t+1} &= d_t+(1-p)a_t - [\delta_t+(1-p) \, \alpha_t]\\[0.2cm]
&= d_t-\delta_t+(1-p)(a_t-\alpha_t) \cpr 0,
\end{align*}}%
as $N \to \infty$.
Since $i_{t+1}+d_{t+1}+a_{t+1}=\iota_{t+1}+\delta_{t+1}+\alpha_{t+1}=1$, we conclude that
$\bbE(a_{t+1})-\alpha_{t+1} \cpr 0$.
\end{proof}

\begin{lem}
\label{L: G4}
For every $t \geq 0$, we have that
\[ \bbV(\eta_{t+1}) \cpr \vzero \quad \text{as } N \to \infty. \]
%Suppose that $\eta_t - \xi_t \cpr \vzero$ as $N \to \infty$, for some $t \geq 0$.
%Then,
\end{lem}

\begin{proof}
The result is a consequence of Lemma~\ref{L: G2}.
First, we claim that
\begin{equation}
\label{F: VarI}
\bbV(i_{t+1}) \leq \dfrac{1}{N+1}.
\end{equation}
To prove \eqref{F: VarI}, we distinguish three cases:
\begin{lprova}
\item If $I_t=0$, then $\bbV(i_{t+1})=0$.

\item If $I_t=1$, then
\begin{equation*}
\bbV(i_{t+1}) = \left(\dfrac{1}{N+1}\right)^2\left[\left(1-\dfrac{p}{N}\right)^{A_t}-\left(1-\dfrac{p}{N}\right)^{2 A_t}\right]
\leq \left(\dfrac{1}{N+1}\right)^2.
\end{equation*}

\item If $I_t\geq 2$, then
{\allowdisplaybreaks
\begin{align*}
\bbV(i_{t+1})&= i_t \, \biggl(\left(i_t-\dfrac{1}{N+1}\right)\left(1-\dfrac{2p}{N}\right)^{A_t}-i_t\left(1-\dfrac{p}{N}\right)^{2A_t}+\dfrac{1}{N+1}\left(1-\dfrac{p}{N}\right)^{A_t}\biggr)\\[0.2cm]
&\leq i_t \, \biggl(i_t \, \biggl[\left(1-\dfrac{2p}{N}\right)^{A_t}-\left(1-\dfrac{p}{N}\right)^{2A_t}\biggr]+\dfrac{1}{N+1}\biggr).
\end{align*}}%
Applying Bernoulli's inequality, we conclude that the expression between brackets is nonpositive, so
\[ \bbV(i_{t+1}) \leq \dfrac{i_t}{N+1}\leq \dfrac{1}{N+1}. \]
\end{lprova}
Thus, \eqref{F: VarI} holds, and, from this equation, it follows that $\bbV(i_{t+1}) \cpr 0$ as $N \to \infty$.
Furthermore,
\begin{equation*}
\bbV(d_{t+1})=a_t \, \dfrac{p(1-p)}{N+1}\leq \dfrac{p(1-p)}{N+1} \cpr 0,
\end{equation*}
as $N \to \infty$.
Finally, we observe that
{\allowdisplaybreaks
\begin{align*}
\bbV(a_{t+1})&= \bbV(i_{t+1}) + \bbV(d_{t+1}) + 2 \, p \, a_t \, i_t
\left(1-\dfrac{p}{N}\right)^{A_t}\left(\dfrac{1-p}{N-p}\right)\\[0.2cm]
&\leq \bbV(i_{t+1}) + \bbV(d_{t+1}) + 2 \, p \left(\dfrac{1-p}{N-p}\right),
\end{align*}}%
whence $\bbV(a_{t+1}) \cpr 0$.
\end{proof}

\begin{proof}[Proof of Theorem~\ref{T: LLN-G}]
We proceed by induction.
For $t = 0$, the statement is trivially true, as the initial conditions of the stochastic model and the deterministic system coincide.
Now, let $t \geq 0$, and assume~\eqref{F: LLN-G} is true.
By Lemma~\ref{L: G3}, we have that $\bbE(i_{t+1})-\iota_{t+1} \cpr 0$ as $N \to \infty$, hence the dominated convergence theorem yields
{\allowdisplaybreaks
\begin{align*}
E(i_{t+1})-\iota_{t+1} & \rightarrow 0, \quad \text{and}\\[0.2cm]
\V(\bbE(i_{t+1})) & \rightarrow 0.
\end{align*}}%
Thus, using Lemma~\ref{L: G4} and the dominated convergence theorem again, we get
\begin{equation*}
\V(i_{t+1}) = E(\bbV(i_{t+1})) + \V(\bbE(i_{t+1})) \rightarrow 0.
\end{equation*}
Therefore, by Markov inequality, for every $\eps > 0$, we have that
\begin{equation*}
P(|i_{t+1}-\iota_{t+1}|>\eps) \leq 
\dfrac{1}{\eps^2} \left(\V(i_{t+1})+\left[E(i_{t+1})-\iota_{t+1}\right]^2 \right)
\rightarrow 0.
\end{equation*}
So $i_{t+1}-\iota_{t+1} \cpr 0$ as $N \to \infty$.
The same arguments hold for the other components of the random vector $\eta_{t+1} - \xi_{t+1}$, thereby completing the induction step.
\end{proof}

\subsection{Proof of Theorem~\ref{T: ABG}}
\label{SS: Proof ABG}

Fixed $p \in (0, 1)$ and $N \geq 3$, recall that $\xi_t = (\iota_t,\alpha_t,\delta_t)$ satisfies the equations
\begin{subequations}
\label{F: SD-G}
{\allowdisplaybreaks
\begin{alignat}{2}
&\iota_{t+1} = \iota_t \, e^{-p \alpha_t},\label{F: SD-I}\\[0.2cm]
&\alpha_{t+1} = p \, \alpha_t+\iota_t \, (1-e^{-p\alpha_t}),\label{F: SD-A}\\[0.2cm]
&\delta_{t+1} = \delta_t+(1-p) \, \alpha_t,\label{F: SD-D}\\[0.2cm]
&\iota_0 = \dfrac{N}{N+1}, \; \alpha_0 = \dfrac{1}{N+1}, \; \delta_0 = 0.\nonumber
\end{alignat}}%
\end{subequations}
Notice that $\iota_t$, $\alpha_t$ and $\delta_t$ are positive for every $t \geq 1$, and 
$\iota_t + \alpha_t + \delta_t = 1$ for every $t \geq 0$.
Moreover, as $\{\iota_t\}$ is decreasing and $\{\delta_t\}$ is increasing in $t$, the following limits exist:
\begin{equation*}
\iota_{\infty}^{(N)}=\lim_{t\to \infty}\iota_t^{(N)}
\quad \text{and} \quad
\delta_{\infty}^{(N)}=\lim_{t\to \infty}\delta_t^{(N)}.
\end{equation*}
From~\eqref{F: SD-D}, it follows that $\displaystyle\lim_{t\to \infty }\alpha_t^{(N)}=0$, thus $\iota_{\infty}^{(N)}+\delta_{\infty}^{(N)}=1$.

Now from \eqref{F: SD-I}, we obtain that for $t \geq 1$,
\begin{equation*}
\iota_t=\iota_0 \, \exp\biggl\{-p\sum_{j=0}^{t-1}\alpha_j\biggr\}.
\end{equation*}
However, from \eqref{F: SD-D},
\begin{equation*}
\sum_{j=0}^{t-1}\alpha_{j}=\dfrac{1}{1-p}\sum_{j=0}^{t-1}\left(\delta_{j+1}-\delta_{j}\right)=\dfrac{\delta_t}{1-p}.
\end{equation*}
Therefore,
\begin{equation*}
\iota_t=\iota_0 \, \exp\left\{-\phi(p) \, \delta_t\right\}, \, t \geq 0.
\end{equation*}
Taking $t \to \infty$, we conclude that $\iota_{\infty}^{(N)}$ is the unique fixed point of the function
\begin{equation*}
\tau^{(N)}(x)=\left(\dfrac{N}{N+1}\right) \exp\{-\phi(p) (1 - x)\}
\end{equation*}
in the interval [0, 1].

Since $\tau^{(N)}(x) \leq \tau^{(N+1)}(x)$ for every $x$, we have that $\iota_{\infty}^{(N)} \leq \iota_{\infty}^{(N+1)}$, hence there exists $\iota_{\infty}=\displaystyle\lim_{N\to \infty} \iota_{\infty}^{(N)}$.
Next, we observe that the sequence of functions $\{ \tau^{(N)} \}$ converges uniformly as $N \to \infty$ to the function $\tau$ defined in \eqref{F: Func}.
Consequently, $\iota_{\infty}$ is a fixed point of the function $\tau$ in $[0, 1]$.

Now, if $p\leq 1/2$, then $\tau^{\prime}(1)=\phi(p)\leq 1$, so $\iota_{\infty}=1$ is the unique fixed point of~$\tau$ in $[0, 1]$.
On the other hand, if $p> 1/2$, then $\tau^{\prime}(1)=\phi(p)> 1$, and $\tau$ has two fixed points, one at $x_1 = -[\phi(p)]^{-1} \, W_0(-\phi(p) \, e^{-\phi(p)}) < 1$, and the other at $x_2 = 1$.
To finish the proof, it is enough to show that $\iota_{\infty} < 1$ for $p> 1/2$.
To accomplish this, we use that, for every $x \in (0, 1)$,
\begin{equation}
\label{F: Des}
-\dfrac{x(1+x)}{1-x} \leq \log(1-x).
\end{equation}
By taking $x=2p-1$ in \eqref{F: Des}, we obtain that
\[ \tau^{(N)}(2(1-p)) < \exp\left\{ -\dfrac{p}{1 - p} \, (2 p - 1) \right\} \leq 2 (1-p), \]
so that $\iota_{\infty}^{(N)}<2(1-p)$.
Hence, $\iota_{\infty}\leq 2(1-p) <1$.~\qed

\subsection{Proof of Theorem~\ref{T: LLN-NG}}
\label{SS: Proof LLN-NG}

We follow a strategy similar to the one used in the proof of Theorem~\ref{T: LLN-G}.
As we observe in Section~\ref{SS: NGM}, for the nongeometric model, the Markov chain $\{(I_t, A_t, D_t)\}_{t \geq 0}$ has transitions described by the pair of formulas \eqref{F: Aux-NG}--\eqref{F: Trans-NG}, which are essentially different from \eqref{F: Aux-G}--\eqref{F: Trans-G}, set out for the geometric model.
This entails that other computations are needed in establishing the analogues of Lemmas~\ref{L: G1} to~\ref{L: G4}.
We include the details here, mainly for the sake of completeness.
For $t \geq 0$, let 
\(\Ft = \sigma(\{\left(I_s, A_s, D_s\right): 0 \leq s \leq t \})\).

\begin{lem}
\label{L: NG1}
For every $t \geq 0$, the conditional expectations of $I_{t+1}$, $A_{t+1}$ and $D_{t+1}$ given $\Ft$ are, respectively,
{\allowdisplaybreaks
\begin{align*}
\bbE(I_{t+1})&= I_t\left(1-\dfrac{1}{N}\right)^{A_t},\\[0.2cm]
\bbE(A_{t+1})&= I_t\left[\dfrac{A_t}{N}+1-\left(1-\dfrac{1}{N}\right)^{A_t}\right],\\[0.2cm]
\bbE(D_{t+1})&= D_t+\left(1-\dfrac{I_t}{N}\right)A_t.
\end{align*}}%
\end{lem}

\begin{lem}
\label{L: NG2}
For every $t \geq 0$, the conditional variances of $I_{t+1}$, $D_{t+1}$ and $A_{t+1}$ given~$\Ft$ are, respectively,
{\allowdisplaybreaks
\begin{align*}
\bbV(I_{t+1})&= I_t \, \biggl(\left(I_t-1\right)\left(1-\dfrac{2}{N}\right)^{A_t}-I_t\left(1-\dfrac{1}{N}\right)^{2A_t}+\left(1-\dfrac{1}{N}\right)^{A_t}\biggr),\\[0.2cm]
\bbV(D_{t+1})&= A_t\left(\dfrac{I_t}{N}\right)\left(1-\dfrac{I_t}{N}\right),\\[0.1cm]
\bbV(A_{t+1})&= \bbV(I_{t+1}) + \bbV(D_{t+1}) + 
\dfrac{2 \, A_t \, I_t}{N}\left(1-\dfrac{1}{N}\right)^{A_t}\left(\dfrac{N-I_t}{N-1}\right).
\end{align*}}%
\end{lem}

\begin{lem}
\label{L: NG3}
Suppose that $\eta_t - \tilde{\xi}_t \cpr \vzero$ as $N \to \infty$, for some $t \geq 0$.
Then,
\[ \bbE(\eta_{t+1}) - \tilde{\xi}_{t+1} \cpr \vzero \quad \text{as } N \to \infty. \]
\end{lem}

\begin{lem}
\label{L: NG4}
For every $t \geq 0$, we have that
\[ \bbV(\eta_{t+1}) \cpr \vzero \quad \text{as } N \to \infty. \]
%Suppose that $\eta_t - \tilde{\xi}_t \cpr \vzero$ as $N \to \infty$, for some $t \geq 0$.
%Then,
\end{lem}

Using Lemmas~\ref{L: NG3} and \ref{L: NG4}, Theorem~\ref{T: LLN-NG} is shown by induction on $t$, as we have done for Theorem~\ref{T: LLN-G}.

\begin{proof}[Proof of Lemma~\ref{L: NG1}]
From~\eqref{F: Aux-NG} and \eqref{F: FGP-Bin1}, we have that, for every $s \in \bbR$,
\begin{equation}
\label{F: FGM-Z-NG}
\bbE(s^{Z_{t+1}}) = \left(1-\dfrac{I_t}{N}(1-s)\right)^{A_{t}}.
\end{equation}
But from~\eqref{F: EspEB}, 
\begin{equation*}
\bbE(I_{t+1}|Z_{t+1}) = I_t\left(\dfrac{I_t-1}{I_t}\right)^{Z_{t+1}}.
\end{equation*}
Taking the expectation and using \eqref{F: FGM-Z-NG} gives
\begin{equation*}
\bbE(I_{t+1}) = I_t\left(1-\dfrac{I_t}{N}\cdot\dfrac{1}{I_t}\right)^{A_t}
= I_t\left(1-\dfrac{1}{N}\right)^{A_t}.
\end{equation*}
The formulas for the conditional expectations of $A_{t+1}$ and $D_{t+1}$ follow from the facts that $A_{t+1} = Z_{t+1}+I_t-I_{t+1}$ and $D_{t+1} = N+1-A_{t+1}-I_{t+1}$.
\end{proof}

\begin{proof}[Proof of Lemma~\ref{L: NG2}]
Using~\eqref{F: EspEB} and \eqref{F: VarEB}, we obtain
\begin{gather*}
\bbE\left(I_{t+1}|Z_{t+1}\right) = I_t\left(\dfrac{I_t-1}{I_t}\right)^{Z_{t+1}} \, \text{ and}\\[0.2cm]
\bbV(I_{t+1}|Z_{t+1}) = I_t(I_t-1)\left(\dfrac{I_t-2}{I_t}\right)^{Z_{t+1}}+I_t\left(\dfrac{I_t-1}{I_t}\right)^{Z_{t+1}}-I_t^2\left(\dfrac{I_t-1}{I_t}\right)^{2Z_{t+1}}.
\end{gather*}
Thus, the law of total variance and \eqref{F: FGM-Z-NG} yield
{\allowdisplaybreaks
\begin{align*}
\bbV(I_{t+1})&= I_t(I_t-1)\left(1-\dfrac{2}{N}\right)^{A_t}+I_t\left(1-\dfrac{1}{N}\right)^{A_t}-I_t^2 \, E\biggl(\left(\dfrac{I_t-1}{I_t}\right)^{2Z_{t+1}}\biggm|\Ft\biggr)+{}\\[0.2cm]
&\phantom{=}{}+I_t^2 \, E\biggl(\left(\dfrac{I_t-1}{I_t}\right)^{2Z_{t+1}}\biggm|\Ft\biggr)-I_t^2\left(1-\dfrac{1}{N}\right)^{2A_t}\\[0.2cm]
&= I_t(I_t-1)\left(1-\dfrac{2}{N}\right)^{A_t}-I_t^2\left(1-\dfrac{1}{N}\right)^{2A_t}
+I_t\left(1-\dfrac{1}{N}\right)^{A_t}.
\end{align*}}%

Since $D_{t+1} = N+1-Z_{t+1}-I_t$, we have that
\[ \bbV(D_{t+1}) = \bbV(Z_{t+1}) = A_t\left(\dfrac{I_t}{N}\right)\left(1-\dfrac{I_t}{N}\right). \]

Now, as $A_{t+1} = Z_{t+1}+I_t-I_{t+1}$, we get
\begin{equation}
\label{F: VarA-NG}
\bbV(A_{t+1})=\bbV(I_{t+1})+\bbV(D_{t+1})- 2 \, \bbC(I_{t+1}, Z_{t+1}).
\end{equation}
But, using formula \eqref{F: FGP-Bin2},
{\allowdisplaybreaks
\begin{align*}
\bbE(I_{t+1} Z_{t+1})&= \bbE(Z_{t+1} \, \bbE(I_{t+1}|Z_{t+1}))\\[0.2cm]
&= \bbE\biggl(Z_{t+1} \, I_t\left(\dfrac{I_t-1}{I_t}\right)^{Z_{t+1}}\biggr)
= \frac{A_t \, I_t \, (I_t - 1)}{N} \left(1-\dfrac{1}{N}\right)^{A_t-1}.
\end{align*}}%
Therefore,
{\allowdisplaybreaks
\begin{align*}
\bbC(I_{t+1}, Z_{t+1}) &= \bbE(I_{t+1} Z_{t+1}) - \bbE(I_{t+1}) \, \bbE(Z_{t+1})\\[0.2cm]
&= \frac{A_t \, I_t}{N} \left(1-\dfrac{1}{N}\right)^{A_t}
\left[(I_t - 1) \left(\frac{N}{N - 1}\right) - I_t\right]\\[0.2cm]
&= \dfrac{A_t \, I_t}{N}\left(1-\dfrac{1}{N}\right)^{A_t}\left(\dfrac{I_t-N}{N-1}\right).
\end{align*}}%
Substituting this in \eqref{F: VarA-NG} gives the desired formula for $\bbV(A_{t+1})$.
\end{proof}

\begin{proof}[Proof of Lemma~\ref{L: NG3}]
Suppose that $\eta_t-\tilde{\xi}_t \cpr \vzero$ for some $t \geq 0$.
By following the same steps as in the proof of \eqref{F: ConvI-G} (Lemma~\ref{L: G3}) with $p = 1$, we obtain that $\bbE(i_{t+1})-\tilde{\iota}_{t+1} \cpr 0$ as $N \to \infty$.
From Lemma~\ref{L: NG1} and formula \eqref{F: SDNG},
\begin{equation*}
\bbE(d_{t+1})-\tilde{\delta}_{t+1}=d_t-\tilde{\delta}_t+a_t\left[1-\left(\dfrac{N+1}{N}\right)i_t\right]-\tilde{\alpha}_t(1-\tilde{\iota}_t).
\end{equation*}
Then, applying the method used in establishing \eqref{F: ConvI-G}, with the function
\[ \varphi_N(i,a)=a\left[1-\left(\dfrac{N+1}{N}\right)i\right], \, (i,a)\in [0, 1]^2, \]
which is Lipschitz continuous with Lipschitz constant at most~$\sqrt{5}$, we conclude that
\(\bbE(d_{t+1})-\tilde{\delta}_{t+1} \cpr 0\) as $N \to \infty$.
Using that $i_{t+1}+d_{t+1}+a_{t+1}=\tilde{\iota}_{t+1}+\tilde{\delta}_{t+1}+\tilde{\alpha}_{t+1}=1$, we deduce that
$\bbE(a_{t+1})-\tilde{\alpha}_{t+1} \cpr 0$.
\end{proof}

\begin{proof}[Proof of Lemma~\ref{L: NG4}]
Notice that, from Lemma~\ref{L: NG2},
{\allowdisplaybreaks
\begin{align*}
\bbV(i_{t+1})&= i_t \, \biggl(\left(i_t-\dfrac{1}{N+1}\right)\left(1-\dfrac{2}{N}\right)^{A_t}-i_t\left(1-\dfrac{1}{N}\right)^{2A_t}+\dfrac{1}{N+1}\left(1-\dfrac{1}{N}\right)^{A_t}\biggr),\\[0.2cm]
\bbV(d_{t+1})&= \dfrac{a_t \, i_t}{N} \left(1-\dfrac{I_t}{N}\right),\\[0.1cm]
\bbV(a_{t+1})&= \bbV(i_{t+1}) + \bbV(d_{t+1}) + 
\dfrac{2 \, a_t \, i_t}{N}\left(1-\dfrac{1}{N}\right)^{A_t}\left(\dfrac{N-I_t}{N-1}\right).
\end{align*}}%
We prove that $\bbV(\eta_{t+1}) \cpr \vzero$ as $N \to \infty$, by following a similar line of arguments as in the proof of Lemma~\ref{L: G4}.
\end{proof}

%\section*{Acknowledgements}

\bibliography{Bib_LLN}
\bibliographystyle{plainnat}

\end{document}